\documentclass[12pt,fleqn,oneside]{article}
\usepackage[inner=32mm,outer=31mm,tmargin=30mm,bmargin=40mm]{geometry}
\usepackage[utf8]{inputenc}
\usepackage[T1]{fontenc}
\usepackage{amssymb}
\usepackage{hyperref}
\usepackage{geometry}
\usepackage{amsmath}
\usepackage{amsthm}
\usepackage{xfrac}
\usepackage{graphicx}
\usepackage{color}
\usepackage[normalem]{ulem}
\usepackage[english]{babel}
\usepackage[onehalfspacing]{setspace}
\usepackage{verbatim}

\usepackage{rotating}

\newcommand{\IR}{\ensuremath{\mathbb{R}}}
\newcommand{\IN}{\ensuremath{\mathbb{N}}}
\newcommand{\IZ}{\ensuremath{\mathbb{Z}}}
\newcommand{\IQ}{\ensuremath{\mathbb{Q}}}

\newcommand{\IE}{\ensuremath{\mathbb{E}}}

\newcommand{\Hmix}{\ensuremath{H^{r,{\rm mix}}(\IR^d)}}
\newcommand{\Hmixunit}{\ensuremath{{H^{r,{\rm mix}}([0,1]^d)}}}
\newcommand{\Hmixcomp}{\ensuremath{\mathring{H}^{r,{\rm mix}}(\IR^d)}}
\newcommand{\Hmixunitcomp}{\ensuremath{\mathring{H}^{r,{\rm mix}}([0,1]^d)}}
\newcommand{\Hiso}{\ensuremath{H^s(\IR^d)}}
\newcommand{\Hisounit}{\ensuremath{{H^s([0,1]^d)}}}
\newcommand{\Hisocomp}{\ensuremath{\mathring{H}^s(\IR^d)}}
\newcommand{\Hisounitcomp}{\ensuremath{\mathring{H}^s([0,1]^d)}}
\newcommand{\Cc}{\ensuremath{C_c(\IR^d)}}

\newcommand{\supp}{\mathop{\mathrm{supp}}}
\newcommand{\diag}{\mathop{\mathrm{diag}}}
\newcommand{\mixnorm}[1]{\left\Vert#1\right\Vert_{{H^{r,{\rm mix}}(\IR^d)}}} 
\newcommand{\mixscalar}[2]{{\left\langle#1,#2\right\rangle_{H^{r,{\rm mix}}(\IR^d)}}}
\newcommand{\isonorm}[1]{\left\Vert#1\right\Vert_{{H^s}(\IR^d)}}
\newcommand{\isoscalar}[2]{\left\langle#1,#2\right\rangle_{{H^s(\IR^d)}}}
\newcommand{\lnorm}[1]{\left\Vert#1\right\Vert_{L^2(\IR^d)}}
\newcommand{\llnorm}[1]{\left\Vert#1\right\Vert_{L^2([0,1]^d)}}
\newcommand{\lscalar}[2]{\left\langle#1,#2\right\rangle_{L^2(\IR^d)}}
\newcommand{\llscalar}[2]{\left\langle#1,#2\right\rangle_{L^2([0,1]^d)}}
\newcommand{\set}[1]{\left\{#1\right\}}
\newcommand{\abs}[1]{\left|#1\right|}
\newcommand{\braces}[1]{\left(#1\right)}
\renewcommand{\d}{{\rm d}} 
\newcommand{\scalar}[2]{\left\langle#1,#2\right\rangle}
\newcommand{\mixnormunit}[1]{\left\Vert#1\right\Vert_{{H^{r,{\rm mix}}([0,1]^d)}}} 
\newcommand{\mixscalarunit}[2]{{\left\langle#1,#2\right\rangle_{H^{r,{\rm mix}}([0,1]^d)}}}
\newcommand{\isonormunit}[1]{\left\Vert#1\right\Vert_{{H^s}([0,1]^d)}}
\newcommand{\isoscalarunit}[2]{\left\langle#1,#2\right\rangle_{{H^s([0,1]^d)}}}

\newcommand{\edit}[1]{{#1}}

\newcommand{\diff}{D} 

\theoremstyle{break}
\newtheorem{thm}{Theorem}
\newtheorem*{summary}{Summary}

\theoremstyle{plain}
\newtheorem{lemma}{Lemma}

\title{A Universal Algorithm for Multivariate Integration} 
 
\author{David Krieg\\
Mathematisches Institut, Universit\"at Jena\\ 
Ernst-Abbe-Platz 2, 07743 Jena, Germany\\ 
email: david.krieg@uni-jena.de\\
and\\ 
Erich Novak
\\ 
Mathematisches Institut, Universit\"at Jena\\ 
Ernst-Abbe-Platz 2, 07743 Jena, Germany\\ 
email: erich.novak@uni-jena.de}

\date{\edit{January 30, 2016}}  

\begin{document}

\maketitle

\begin{abstract} 
We present an algorithm for multivariate integration over cubes 
that is unbiased and has optimal order of convergence (in the randomized 
sense as well as in the worst case setting) for all 
Sobolev spaces $\Hmixunit$ and $\Hisounit$ for $s>d/2$. 
\end{abstract}


\medskip\noindent 
Short title: A Universal Algorithm for Integration


\medskip\noindent 
Communicated by Andrew Stuart 

\medskip\noindent
Corresponding author: David Krieg 

\section{Introduction} 

We present a new algorithm
$$
A_n(f) = \sum_{i=1}^n a_i f(x_i) 
$$
for the approximation of integrals 
$$
I_d(f) = \int_{[0,1]^d} f(x) \, \d x . 
$$

Fred Hickernell wrote a paper ``My dream quadrature rule'' 
where he proposed five criteria that an ideal or ``dream'' 
quadrature formula should satisfy.
We also present a list of five (similar, but different) properties 
of our ``dream algorithm'': 

\begin{itemize} 

\item[{\bf (P1)}]
The algorithm $A_n$ should be an unbiased randomized algorithm, i.e., 
$$
\IE (A_n(f)) = I_d(f)
$$
for all integrable functions. 
Of course this means that the weights $a_i \in \IR$ and the points $x_i
\in [0,1]^d$ are random variables. 
It is beneficial to have positive weights $a_i \ge 0$ for all $i$.

\item[{\bf (P2)}]  
The randomized error 
$$
\IE (|A_n(f) - I_d(f)|)
$$
of $A_n$ should be small and/or optimal 
in the sense of order of convergence for ``many'' different 
classes of functions. 
In particular, we would like to have 
\begin{equation}   \label{upperbound1} 
\IE (|A_n(f) - I_d(f)|)  \le \edit{c_{r,d}} \, n^{-r-1/2} \, (\log n)^{(d-1)/2} \, 
\Vert f \Vert_\Hmixunit 
\end{equation} 
\edit{for all $r \in \IN$,} as well as for all $s \in \IN$ with  $s > d/2$
\begin{equation}   \label{upperbound2} 
\IE (|A_n(f) - I_d(f)|)  \le \edit{c_{s,d}} \,  n^{-s/d-1/2} \,  
\Vert f \Vert_\Hisounit  . 
\end{equation} 

\item[{\bf (P3)}] 
The worst case error 
$$
\sup_\omega  |A^\omega_n(f) - I_d(f)|
$$
\edit{among the realizations $A^\omega_n$ of $A_n$ should be}
small and/or optimal 
in the sense of order of convergence for ``many'' different 
classes of functions, \edit{in particular}
\begin{equation}   \label{upperbound3} 
\sup_\omega  (|A^\omega_n(f) - I_d(f)|)  \le \edit{c_{r,d}} \,  n^{-r} \, (\log n)^{(d-1)/2} \, 
\Vert f \Vert_\Hmixunit 
\end{equation} 
\edit{for all $r \in \IN$,} as well as for all $s \in \IN$ with  $s>d/2$
\begin{equation}   \label{upperbound4} 
\sup_\omega  (|A^\omega_n(f) - I_d(f)|)  \le \edit{c_{s,d}} \,  n^{-s/d} \,  
\Vert f \Vert_\Hisounit .
\end{equation} 

\item[{\bf (P4)}] 
The algorithm should have good tractability properties
in the sense of the theory of ``tractability of multivariate 
problems'', see \cite{novwoz}. 

\item[{\bf (P5)}] 
The algorithm should be easy to implement. 

\end{itemize} 

In this paper 
we concentrate on properties 
{\bf (P1)},  {\bf (P2)} and {\bf (P3)} and hence we are not specific on 
{\bf (P4)}  and {\bf (P5)}  and 
leave them for further research.
\edit{In particular, we do not discuss tractability and all
constants $c>0$ may depend on the dimension $d$ and the smoothness $r$ or $s$.}
A few remarks are in order. 

\begin{enumerate} 

\item 
The simplest Monte Carlo method certainly satisfies {\bf (P1)}. 
Therefore it is easy to run the algorithm a few times 
and to do an (a posteriori) error analysis. 
This is a great advantage of an unbiased algorithm. 
Of course the low rate $n^{-1/2}$ (even for very smooth integrands) 
is a big disadvantage of the simplest Monte Carlo method. 
Randomized algorithms with a higher rate of convergence 
are known and often they are unbiased; usually they are designed 
for a specific class of functions. 

\item 
We do not know of any algorithm in the literature that satisfies 
{\bf (P2)}, even in the univariate case $d=1$.
The upper bound \eqref{upperbound1} seems to be new. 
The main term $n^{-r-1/2}$ is of course optimal. 


The bounds \edit{(2)--(4)} are known 
and it is also known that they are optimal. 
The bound (3) is from Frolov, see \cite{frolov,temlyakov,ullrich}. 
The bounds (2) and (4) are from Bakhvalov and can
be found in \cite{ln}.

\item
Many known algorithms (such as the Gaussian quadrature formulas) 
satisfy {\bf (P3)} in the univariate case. 
It is also known that (modifications of) the Frolov algorithm satisfy
{\bf (P3)} for arbitrary $d$. 
Hence the Frolov algorithm (or some modifications of it) 
is ``universal'' in the worst case setting, see also the recent paper \cite{marioandtino}. 
Since it is a deterministic algorithm it certainly cannot satisfy 
{\bf (P1)}  or {\bf (P2)}. 
The problem with any deterministic algorithm $A_n$ is that 
a computation of $A_n(f)$ does not come together with an 
error bound since usually the norm of $f$ is not known. 

\item
We did not discuss the property ``extensible'' in the list of 
Hickernell. We believe that this is another nice property but 
not as important as the other properties since it can decrease the 
total computing time only \edit{slightly}.

\end{enumerate} 

In this paper we present an algorithm $\widetilde{M}_{a,B}$ with positive weights
that satisfies {\bf (P1)}  and  {\bf (P2)}  and {\bf (P3)}, see Section~\ref{s5}. 
In particular we prove the existence of $A_n$ such that 
\eqref{upperbound1} holds. 

\section{Some Notation} 

For $r,d\in\IN$ the tensor product Sobolev space $\Hmix$ is defined as the space
\[
\Hmix = \set{f\in L^2(\IR^d)\mid \diff^\alpha f 
\in L^2(\IR^d) \text{ for every } \alpha \in \set{0,\dots,r}^d}
\]
of real valued functions, equipped with the scalar product
\[
\mixscalar{f}{g}= \sum\limits_{\alpha \in \set{0,\dots,r}^d} 
\lscalar{\diff^\alpha f}{\diff^\alpha g}
\]
and hence with the norm
\[
\mixnorm{f}= \left(\sum\limits_{\alpha \in \set{0,\dots,r}^d} 
\lnorm{\diff^\alpha f}^2\right)^{1/2} . 
\]
It is known that $\Hmix$ is a Hilbert space and its elements can be taken 
to be continuous functions.
In this paper, the Fourier transform is the unique continuous linear map
$\hat{\cdot}: L_2(\IR^d)\to L_2(\IR^d)$
with
\[
\hat{f}(y) = \int_{\IR^d}   f(x)\, e^{-2\pi i \scalar{x}{y}}  \, \d x
\]
for integrable $f$ and $y\in\IR^d$.
The space $\Hmix$ contains exactly those 
functions $f\in L^2(\IR^d)$ with ${\hat{f} \cdot h_r^{1/2} \in L^2(\IR^d)}$ 
for the Fourier transform $\hat{f}$ of $f$ and the weight function
\[
h_r: \IR^d \to \IR^+,\quad h_r(x)= \sum\limits_{\alpha\in\{0,\dots,r\}^d} 
\prod\limits_{j=1}^{d} |2\pi x_j|^{2\alpha_j} 
= \prod\limits_{j=1}^{d} \sum\limits_{k=0}^{r} |2\pi x_j|^{2k}.
\]
In terms of its Fourier transform, the norm of $f\in\Hmix$ is given by
\[
\mixnorm{f}^2= \int\limits_{\IR^d} \left|\hat{f}(x)\right|^2\cdot h_r(x) \, \d x.
\]

Analogously, for $s,d\in\IN$ the isotropic Sobolev space $\Hiso$ is    the space
\[
\Hiso = \set{f\in L^2(\IR^d)\mid \diff^\alpha f \in L^2(\IR^d) 
\text{ for every } \alpha \in \IN_0^d\text{ with } \Vert\alpha\Vert_1\leq s}
\]
of real valued functions, equipped with the scalar product
\[
\isoscalar{f}{g}= \sum\limits_{\Vert\alpha\Vert_1\leq s} 
\lscalar{\diff^\alpha f}{\diff^\alpha g}
\]
and hence with the norm
\[
\isonorm{f}= \left(\sum\limits_{\Vert\alpha\Vert_1\leq s} 
\lnorm{\diff^\alpha f}^2\right)^{1/2} . 
\]
This also defines a Hilbert space. In the following, let $s>d/2$. 
Then $\Hiso$ also consists of continuous functions, exactly those functions 
$f\in L^2(\IR^d)$ with ${\hat{f} \cdot v_s^{1/2}\in L^2(\IR^d)}$ 
for the Fourier transform $\hat{f}$ of $f$ and the weight function
\[
v_s: \IR^d \to \IR^+,\quad v_s(x)= \sum\limits_{\Vert\alpha\Vert_1\leq s} 
\prod\limits_{j=1}^{d} |2\pi x_j|^{2\alpha_j}
\asymp \left( 1+\Vert x\Vert_2^2\right)^s
.\]
In terms of its Fourier transform, the norm of $f\in\Hiso$ is given by
\[
\isonorm{f}^2= \int\limits_{\IR^d} \left|\hat{f}(x)\right|^2\cdot v_s(x) \, \d x 
.\]

Furthermore let $\Cc$ be the set of all continuous real valued functions 
with compact support in $\IR^d$.

We will first present an unbiased Monte Carlo method for integration 
on $\Cc$ in Section~\ref{s4}.
We will examine its error for the subspaces $\Hmixcomp$ and $\Hisocomp$ of 
functions in $\Hmix$ or $\Hiso$ with compact support. 
This includes an error bound for the classes $\Hmixunitcomp$ and $\Hisounitcomp$ 
of all functions in $\Hmix$ or $\Hiso$ with support in the unit cube $[0,1]^d$. 
These spaces can also be considered as subspaces of the Hilbert space
\[
\Hmixunit= \set{f\in L^2([0,1]^d) \mid \diff^\alpha 
f \in L^2([0,1]^d) \text{ for every } \alpha \in \set{0,\dots,r}^d}
,\]
equipped with the scalar product
\[
\mixscalarunit{f}{g}= \sum\limits_{\alpha \in \set{0,\dots,r}^d} 
\llscalar{\diff^\alpha f}{\diff^\alpha g}  , 
\]
or the Hilbert space
\[
\Hisounit= \set{f\in L^2([0,1]^d) 
\mid \diff^\alpha f \in L^2([0,1]^d) \text{ for } \alpha 
\in \IN_0^d\text{ with } \Vert\alpha\Vert_1\leq s}
,\]
with the scalar product
\[
\isoscalarunit{f}{g}= \sum\limits_{\Vert\alpha\Vert_1\leq s} 
\llscalar{\diff^\alpha f}{\diff^\alpha g}  , 
\]
respectively. 
It turns out that this method for $\Hmixunitcomp$ and $\Hisounitcomp$ 
can be transformed to the full spaces $\Hmixunit$ and $\Hisounit$ without 
loosing its good properties.

\section{The Basic Quadrature Rule $Q_{S,v}$}
\label{basicquadrulesection}

Let $S\in\IR^{d\times d}$ be any \edit{invertible} matrix and $v$ any vector in $\IR^d$. 
At the basis of the Monte Carlo methods to be presented is the
deterministic and linear quadrature rule
$ Q_{S,v}$,
defined by 
\[
Q_{S,v}(f)=\frac{1}{|\det S|} 
\sum\limits_{m\in\IZ^d} f\left(S^{-\top}(m+v)\right)
\]
for any admissible input function $f:\IR^d\to\IR$. This includes all functions $f$
with compact support. For such functions
the sum is actually a finite sum. 
More precisely, $Q_{S,v}$ uses the nodes $S^{-\top}(m+v)$, 
where $m\in\IZ^d$ is a lattice point in the 
\edit{compact set $\braces{S^\top\left(\supp f\right)-v}$ of
Lebesgue measure $\braces{\det(S)\cdot \lambda_d\left(\supp f\right)}$}. 
This volume is the approximate number of nodes of $Q_{S,v}$.

In particular, the number of nodes of $Q_{aS,v}$ for $a\geq 1$ is of order $a^d$. 
The following simple lemma gives an exact upper bound, see~\cite{skriganov} for other bounds.

\begin{lemma}
\label{anlemma}
Suppose $f: \IR^d \to \IR$
is supported in an axis-parallel 
cube of edge length $l>0$. For any \edit{invertible} matrix $S\in\IR^{d\times d}$, 
$v\in\IR^d$ and $a\geq 1$ the quadrature rule $Q_{aS,v}$ uses at most
$
\left(l\cdot\Vert S\Vert_1+1\right)^d\cdot a^d
$
function values of $f$.
\end{lemma}

\begin{proof}
By assumption, $f$ has compact support in $\frac{l}{2}\cdot [-1,1]^d+x_0$ 
for some $x_0\in\IR^d$. The number of function values is bounded by the size of
\[\begin{split}
M&=\set{m\in\IZ^d \mid (aS)^{-\top}(m+v)\in \frac{l}{2}\cdot [-1,1]^d+x_0}\\
&= \set{m\in\IZ^d \mid m+\left(v-aS^\top x_0\right)\in \frac{al}{2}\cdot S^\top [-1,1]^d}
.\end{split}\]
Since $\Vert S^\top x\Vert_\infty\leq \Vert S^\top\Vert_\infty 
=\Vert S\Vert_1$ for $x\in[-1,1]^d$,
\[
M \subseteq \set{m\in\IZ^d \mid m+\left(v-aS^\top x_0\right)\in 
\left[-\frac{al}{2}\Vert S\Vert_1,\frac{al}{2}\Vert S\Vert_1\right]^d}
\]
and $\vert M\vert \leq \left(al\Vert S\Vert_1+1\right)^d$. 
With $1\leq a$ we get the estimate of Lemma~\ref{anlemma}.
\end{proof}

The error of this algorithm for integration on $\Cc$ can be expressed in 
terms of the Fourier transform.

\begin{lemma}
\label{errorlemma}
For any \edit{invertible} matrix $S\in\IR^{d\times d}$, $v\in\IR^d$ and $f\in C_c(\IR^d)$
\[
\left| Q_{S,v}(f)-I_d(f)\right|  
\leq \sum\limits_{m\in\IZ^d\setminus\set{0}} \left| \hat{f}(Sm)\right|
.\]
\end{lemma}

\begin{proof}
The function $g=f\circ S^{-\top}(\cdot +v)$ is continuous with compact support. 
Hence, the Poisson summation formula and an affine linear substitution $x=S^\top y-v$ yield
\[\begin{split}
Q_{S,v}(f)&=\frac{1}{\abs{\det S}}\sum\limits_{m\in\IZ^d} g(m) 
= \frac{1}{\abs{\det S}}\sum\limits_{m\in\IZ^d} \hat{g}(m)\\
&= \frac{1}{\abs{\det S}}\sum\limits_{m\in\IZ^d} 
\int\limits_{\IR^d} f\left(S^{-\top}(x+v)\right)\cdot e^{-2\pi i\langle x,m\rangle}
\, \d x \\
&= \sum\limits_{m\in\IZ^d} \int\limits_{\IR^d} f\left(y\right)
\cdot e^{-2\pi i\langle S^\top y-v,m\rangle} \, \d y \\
&= \sum\limits_{m\in\IZ^d} \hat{f}(Sm)\cdot e^{2\pi i\langle v,m\rangle}
,\end{split}\]
if the latter series converges absolutely, see \cite[pp.\,356]{koch}. If not,
the stated inequality is obvious.
This proves the statement, 
since $I_d(f)=\hat{f}(S\cdot 0)\cdot e^{2\pi i\langle v,0\rangle}$.
\end{proof}

\section{The Method $M_{a,B}$ for Integration on $\Hmixcomp$ and $\Hisocomp$}
\label{s4} 

It is known how to choose $S$ in $Q_{S,v}$ to get a good 
deterministic quadrature rule on $\Hmixunitcomp$. 
Let the matrix $B\in \IR^{d\times d}$ satisfy the following three conditions:
\begin{itemize}
\item[(a)] $B$ is \edit{invertible},
\item[(b)] $\left|\prod\limits_{j=1}^{d}(Bm)_j\right|\geq 1$, 
for any $m\in\IZ^d\setminus\set{0}$,
\item[(c)] For any $x,y\in\IR^d$ the box $[x,y]$ with 
volume $V=\prod\limits_{j=1}^{d}|x_j-y_j|$ contains at most $V+1$ 
lattice points $Bm$, $m\in\IZ^d$,
\end{itemize}
where $[x,y]=\set{z\in\IR^d\mid z_j 
\text{ is inbetween of }x_j\text{ and }y_j\text{ for }j=1,\dots,d}$. 
Such a matrix shall be called a \textit{Frolov matrix}. 
Property (b) says that for $a>0$ every point of the lattice $aB\IZ^d$ 
but zero lies in the set $D_a$ of all $x\in\IR^d$ 
with $\prod_{j=1}^{d}\abs{x_j}\geq a^d$.


\begin{minipage}[h!]{.65\linewidth}
\includegraphics[width=.95\linewidth]{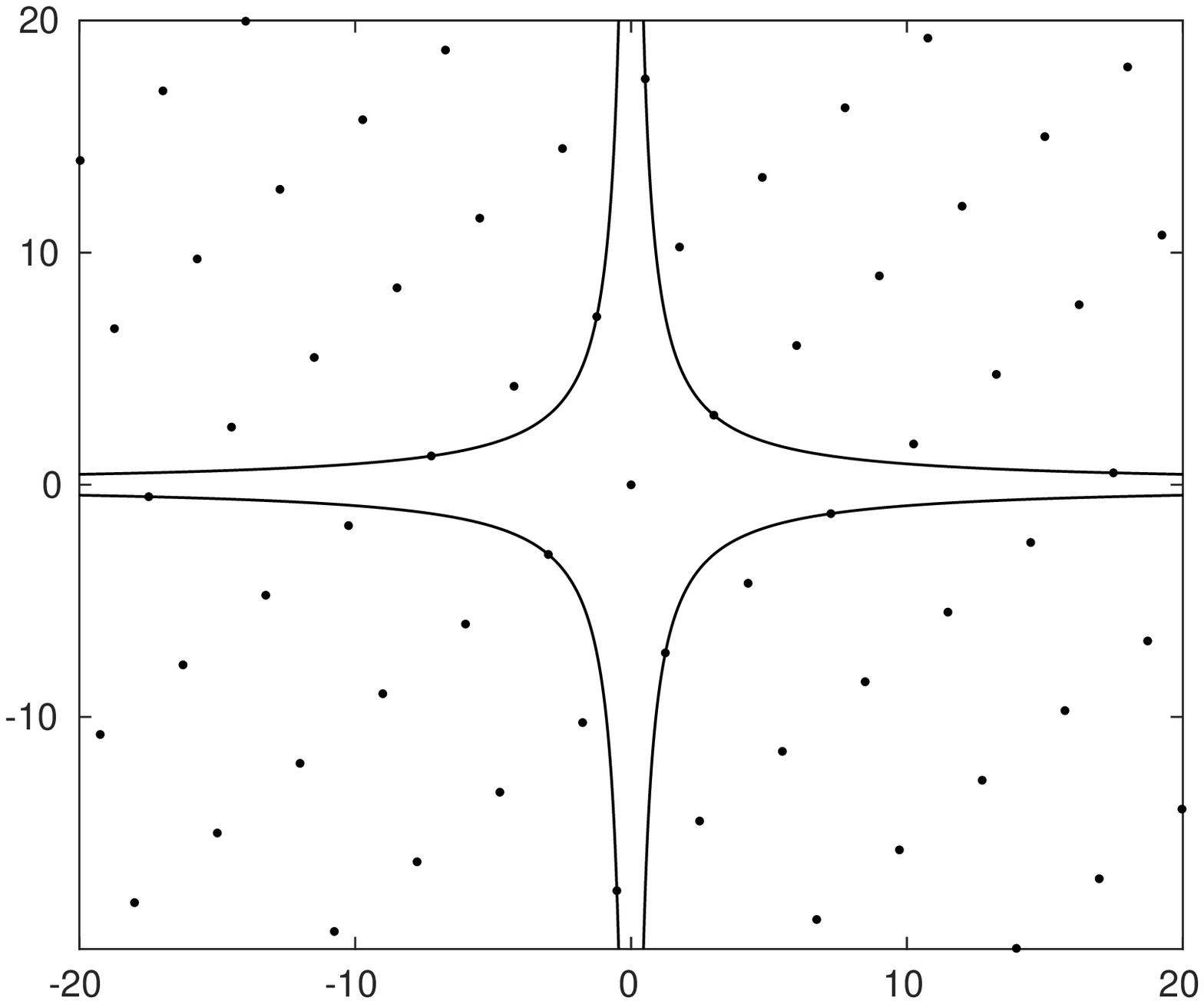}
\end{minipage}
\begin{minipage}[h!]{.35\linewidth}
This graphic shows the lattice $aB\IZ^d$ for $d=2$, $a=3$ and the Frolov matrix
\[B=\begin{pmatrix}
1 & 2-\sqrt{2}\\
1 & 2+\sqrt{2}
\end{pmatrix}
.\]
Except zero, every lattice point lies inside $D_3$.
\end{minipage}

It is known that one can construct such a matrix $B$ in the following way. 
Let $p\in\IZ[x]$ be a polynomial of degree $d$ with leading 
coefficient 1 which is irreducible over $\IQ$ and has $d$ different 
real roots $\zeta_1,\hdots,\zeta_d$. Then the matrix
\[B=\left(\zeta_i^{j-1}\right)_{i,j=1}^d\]
has the desired properties, as shown in \cite[p.\,364]{temlyakovbuch} and \cite{ullrich}. 
In arbitrary dimension $d$ we can choose $p(x)=(x-1)(x-3)\cdot\hdots\cdot(x-2d+1)-1$, 
see \cite{frolov} or \cite{ullrich}, but there are many other possible choices. 
For example, if $d$ is a power of two, we can set $p(x)=2\cos\left(d\cdot\arccos 
(x/2)\right)=2\,T_d(x/2)$, where $T_d$ is the Chebyshev 
polynomial of degree $d$, see \cite[p.\,365]{temlyakovbuch}. 
Then the roots of $p$ are explicitly given by $\zeta_j
=2\cos\left(\frac{2j-1}{2d}\pi\right)$ for $j=1,\hdots,d$.

K.\,K.\,Frolov has already seen in 1976 that the algorithm $Q_{aB,0}$ for $a>1$ 
is optimal on $\Hmixunitcomp$ in the sense of order of convergence. It satisfies
\[
\abs{Q_{aB,0}(f)-I_d(f)} \leq c \,  a^{-rd}\cdot (\log a)^{\frac{d-1}{2}}\cdot \mixnormunit{f}
\]
for a constant $c>0$ and any $a\geq 2$ and $f\in \Hmixunitcomp$. 
We hence call it \emph{Frolov quadrature formula}.
See also \cite{frolov} or \cite{ullrich} for a proof. 
In fact, the same error bound holds for $Q_{aB,v}$ for any $v\in\IR^d$.

\edit{We define a randomized version of this quadrature rule by introducing
two independent random vectors $v$ and $u$. 
With the random shift parameter $v\in[0,1]^d$ 
the algorithm gets unbiased. 
The random dilation parameter $u\in[1,2^{1/d}]^d$ will ensure the general error bound 
of Theorem~1. Both effects are independent of each other: 
The random shift is not needed for the error bound and the dilation is not needed 
for the unbiasedness.} 

\medskip 
\noindent
{\bf Algorithm.} \
For a Frolov matrix $B\in\IR^{d\times d}$ and any $a>0$
the randomized Frolov quadrature formula
$M_{a,B}$
is the method
$Q_{a\bar{u}B,v}$
from Section~\ref{basicquadrulesection} with independent random 
vectors $u$ and $v$, uniformly distributed 
in $[1,2^{1/d}]^d$ and $[0,1]^d$ respectively and 
$\bar{u}=\diag(u_1,\dots,u_d)$.

\medskip 

\begin{lemma}
\label{Munbiased}
The method $M_{a,B}$ is well-defined and unbiased on $L^1(\IR^d)$.
\end{lemma}

\begin{proof}
We realize that for $f\in L^1(\IR^d)$
\[\begin{split}
\IE_u\IE_v \abs{Q_{a\bar{u}B,v}(f)}
&\leq \IE_u\sum\limits_{m\in\IZ^d} \frac{1}{\abs{\det a\bar{u}B}} 
\int_{[0,1]^d} \abs{f\left( (a\bar{u}B)^{-\top}(m+x) \right)} \, \d x \\
&= \IE_u\sum\limits_{m\in\IZ^d}\ \int_{(a\bar{u}B)^{-\top}m+(a\bar{u}B)^{-\top}[0,1]^d} 
\abs{f(y)} \, \d y\,\\
&= \IE_u \int_{\IR^d} \abs{f(y)} \, \d y\,
= \int_{\IR^d} \abs{f(y)} \, \d y\,   < \infty
.\end{split}\]
We can thus apply Fubini's theorem and get
\[\begin{split}
\IE \left(M_{a,B}(f)\right)
&= \IE_u \IE_v \left(Q_{a\bar{u}B,v}(f)\right)\\
&= \IE_u \sum\limits_{m\in\IZ^d} \frac{1}{\abs{\det a\bar{u}B}} 
\int_{[0,1]^d} f\left( (a\bar{u}B)^{-\top}(m+x) \right) \, \d x\\
&= \IE_u \sum\limits_{m\in\IZ^d}\ \int_{(a\bar{u}B)^{-\top}m+(a\bar{u}B)^{-\top}[0,1]^d} 
f(y) \, \d y\\
&= \IE_u \int_{\IR^d} f(y) \, \d y = I_d(f)
.\end{split} \]
In particular, $M_{a,B}(f)$ is almost surely finite.
\end{proof}

According to Lemma~\ref{anlemma} the method $M_{a,B}$ uses no 
more than $2\cdot \left(l\cdot\Vert B\Vert_1+1\right)^d\cdot a^d$ 
function values of a function $f$ supported in a cube of edge length $l$. 
Later we will show that $M_{a,B}$ satisfies
\[
\IE \abs{M_{a,B}(f)-I_d(f)} \leq c \,  a^{-rd-d/2}
\cdot (\log a)^{\frac{d-1}{2}}\cdot \mixnorm{f}
\]
for a constant $c>0$ and any $a\geq 2^{1/d}$ and $f\in \Hmixcomp$. 
But first we analyze $M_{a,B}$ on the larger set $\Cc$.

\subsection*{Error Bound for $\Cc$}

We prove a main result of this paper. 
Again, 
$D_a$ is the set of all $x\in\IR^d$ 
with $\prod_{j=1}^{d}\abs{x_j}\geq a^d$.
The method $M_{a,B}$ satisfies a general error bound on $\Cc$.

\begin{thm}  
\label{keyprop}
Let $B\in\IR^{d\times d}$ be a Frolov matrix. 
Then there is a constant $c>0$ such that for every $a>0$ and $f\in\Cc$
\[\IE \abs{M_{a,B}(f)-I_d(f)} \leq\, c \,  a^{-d} \cdot \int_{D_a} \abs{\hat{f}(x)}\, \d x
.\]
\end{thm} 

\begin{proof}
Let $v\in\IR^d$ be arbitrary, but fixed. Thanks to Lemma~\ref{errorlemma} 
and the monotone convergence theorem we have
\[
\IE_u \abs{Q_{a\bar{u}B,v}(f)-I_d(f)} \leq \IE_u 
\left(\sum\limits_{m\in\IZ^d\setminus\set{0}} \abs{\hat{f}(a\bar{u}Bm)} \right)
= \sum\limits_{m\in\IZ^d\setminus\set{0}} \IE_u \abs{\hat{f}(a\bar{u}Bm)}
.\]
Since each $a\bar{u}Bm$ is uniformly distributed in the 
box $[aBm,2^{1/d}aBm]$ with volume 
$\left(2^{1/d}-1\right)^d\cdot\abs{\prod_{j=1}^d a(Bm)_j}$, this series equals
\[\begin{split}
&\frac{1}{\left(2^{1/d}-1\right)^d} \sum\limits_{m\in\IZ^d\setminus\set{0}}\, 
\int\limits_{[aBm,2^{1/d}aBm]} \frac{\abs{\hat{f}(x)}}{\prod_{j=1}^d
\abs{a(Bm)_j}} \, \d x \\
&\leq \frac{1}{\left(2^{1/d}-1\right)^d} 
\sum\limits_{m\in\IZ^d\setminus\set{0}}\, \int\limits_{[aBm,2^{1/d}aBm]} 
\frac{\abs{\hat{f}(x)}}{\prod_{j=1}^d 2^{-1/d}\abs{x_j}} \, \d x\\
&= \frac{2}{\left(2^{1/d}-1\right)^d}\cdot \int_{\IR^d} 
\frac{\abs{\hat{f}(x)}}{\prod_{j=1}^d \abs{x_j}}\cdot 
\abs{\set{m\in\IZ^d\setminus\set{0}\mid x\in [aBm,2^{1/d}aBm]}} \, \d x\\
&= \frac{2}{\left(2^{1/d}-1\right)^d}\cdot \int_{\IR^d} 
\frac{\abs{\hat{f}(x)}}{\prod_{j=1}^d \abs{x_j}}\cdot 
\abs{\set{m\in\IZ^d\setminus\set{0}\mid Bm\in 
\left[\frac{x}{2^{1/d}a},\frac{x}{a}\right]}}\, \d x  
.\end{split}\]
Thanks to the properties of the Frolov matrix $B$, if $\prod_{j=1}^d \abs{x_j}<a^d$, 
the latter set is empty and otherwise contains no more 
than $\prod_{j=1}^d \abs{\frac{x_j}{a}}+1\leq 2 a^{-d} \prod_{j=1}^d \abs{x_j}$ points.
Thus, we arrive at
\[
\IE_u \abs{Q_{a\bar{u}B,v}(f)-I_d(f)} \leq
\frac{4}{\left(2^{1/d}-1\right)^d}\cdot a^{-d} \int_{D_a} \abs{\hat{f}(x)} \, \d x 
.\]
By Fubini's theorem, we have
\[
\IE \abs{M_{a,B}(f)-I_d(f)}=\IE_v \IE_u \abs{Q_{a\bar{u}B,v}(f)-I_d(f)}
\leq \frac{4}{\left(2^{1/d}-1\right)^d}\cdot a^{-d} \int_{D_a} \abs{\hat{f}(x)} \, \d x 
\]
and the theorem is proven.
\end{proof}

Additional differentiability properties of $f\in\Cc$ result in decay 
properties of $\hat{f}$. This leads to estimates of the 
integral $\int_{D_a} \abs{\hat{f}(x)} \, \d x$.
Hence, the general upper bound for the error of $M_{a,B}(f)$ in Theorem
\ref{keyprop} adjusts to the differentiability of $f$.
Two such examples are functions from $\Hmixcomp$ and $\Hisocomp$.

\subsection*{Error Bounds for $\Hmixcomp$}

If $f\in\Hmixcomp\subseteq\Cc$, the following lemma holds.

\begin{lemma}
\label{intlemmamix}
For any Frolov matrix $B\in\IR^{d\times d}$ and $r\in\IN$ there is some $c>0$ such that
for each $a\geq 2^{1/d}$ and $f\in\Hmixcomp$
\[
\int_{D_a} \abs{\hat{f}(x)} \, \d x  \leq c \,  a^{-rd+d/2}
\, \left(\log a\right)^{\frac{d-1}{2}} \, \mixnorm{f}
.\]
\end{lemma}

\begin{proof}
Applying Hölder's inequality and a linear substitution $x=aBy$ 
to the above integral, we get
\[\begin{split}
&\left( \int_{D_a} \abs{\hat{f}(x)} \, \d x  \right)^2
= \left( \int_{D_a} \abs{\hat{f}(x)} h_r(x)^{1/2} 
\cdot h_r(x)^{-1/2} \, \d x  \right)^2\\
&\leq \mixnorm{f}^2 \cdot \int_{D_a} h_r(x)^{-1} \, \d x 
= \mixnorm{f}^2 \cdot \left(\int_{G} h_r(aBy)^{-1} \, \d y \right) \cdot a^d \abs{\det B}
,\end{split}\]
where $G=B^{-1}D_1$ is the set of all $y\in\IR^d$ with 
$\prod_{j=1}^{d}\abs{(By)_j}\geq 1$. It it thus sufficient to prove that the 
integral $\int_{G} h_r(aBy)^{-1} \, \d y$ is bounded by a constant 
multiple of $a^{-2rd}\cdot \left(\log a\right)^{d-1}$.\\
Consider the auxiliary set $N(\beta)=\{x\in\IR^d \mid \edit{\lfloor2^{\beta_j-1}\rfloor}
\leq|x_j|<2^{\beta_j},1\leq j\leq d\}$ for $\beta\in\IN_0^d$. 
Let $|\beta|=\sum_{j=1}^d \beta_j$ and $G_a^\beta=G\cap\set{y\in\IR^d
\mid aBy\in N(\beta)}$. Since $\IR^d$ is the disjoint union of all $N(\beta)$, 
$G$ is the disjoint union of all $G_a^\beta$ over $\beta\in\IN_0^d$.
For $y\in G_a^\beta$ we have both $\abs{\prod_{j=1}^d a(By)_j}\geq a^d$, 
since $y\in G$, and $\abs{\prod_{j=1}^d a(By)_j}<2^{|\beta|}$, since $aBy\in N(\beta)$. 
This implies $G_a^\beta=\emptyset$ for $|\beta|\leq d\log_2 a$, 
since then $2^{|\beta|}\leq a^d$.

Let $y\in G_a^\beta$ and $|\beta|> d\log_2 a$. Then
\[
h_r(aBy)\geq \prod_{j=1}^d \left(1+\edit{\lfloor2^{\beta_j-1}\rfloor}^{2r}\right) 
\geq \prod_{j=1}^d 2^{2r(\beta_j-1)} = 2^{2r(|\beta|-d)}
\]
and hence $h_r(aBy)^{-1}\leq 2^{2r(d-|\beta|)}$. On the other hand
\[\begin{split}
&\lambda_d(G_a^\beta)\leq \lambda_d\left((aB)^{-1}N(\beta)\right) 
= a^{-d}\cdot |\det B|^{-1}\cdot \lambda_d(N(\beta)) \\
&= a^{-d}\cdot |\det B|^{-1}\cdot 2^d \cdot \prod_{j=1}^d\left(2^{\beta_j}-
\edit{\lfloor2^{\beta_j-1}\rfloor}\right)
\leq a^{-d}\cdot |\det B|^{-1}\cdot 2^d\cdot 2^{|\beta|}
.\end{split}\]
Together we obtain
\[\begin{split}
&\int\limits_{G} h_r(aBy)^{-1} \, \d y 
= \sum\limits_{\beta\in\IN_0^d}\, \int_{G_a^\beta} h_r(aBy)^{-1} \, \d y \\
&= \sum\limits_{|\beta|>d\log_2 a}\ \int_{G_a^\beta} h_r(aBy)^{-1} \, \d y \\
&\leq \sum\limits_{|\beta|>d\log_2 a} 2^{2r(d-|\beta|)}
\cdot a^{-d}\cdot |\det B|^{-1}\cdot 2^d\cdot 2^{|\beta|}\\
&\leq 2^{2rd+d} |\det B|^{-1} \cdot a^{-d} \sum\limits_{k=\lceil d\log_2 a\rceil}^{\infty}
2^{(1-2r)k}\cdot \abs{\set{\beta\in\IN_0^d\mid |\beta|=k}}\\
&\leq 2^{2rd+d} |\det B|^{-1} \cdot a^{-d} \sum\limits_{k=\lceil d\log_2 a\rceil}^{\infty}
2^{(1-2r)k}\cdot (k+1)^{d-1}\\
&= 2^{2rd+d} |\det B|^{-1} \cdot a^{-d} 
\sum\limits_{k=0}^{\infty} 2^{(1-2r)(k+\lceil d\log_2 a\rceil)}\cdot 
\left(k+1+\lceil d\log_2 a\rceil\right)^{d-1}\\
&\leq 2^{2rd+d} |\det B|^{-1} \cdot a^{-d} \cdot a^{(1-2r)d} 
\cdot \sum\limits_{k=0}^{\infty} 2^{(1-2r)k}\cdot 2^{d-1} 
\cdot (k+1)^{d-1}\cdot\lceil d\log_2 a\rceil^{d-1}\\
&\overset{d\log_2a\geq 1}{\leq} 2^{2rd+2d-1} |\det B|^{-1} \cdot a^{-2rd} 
\cdot \sum\limits_{k=0}^{\infty} 2^{(1-2r)k} (k+1)^{d-1} 
\left( 2d\cdot\frac{\log a}{\log 2}\right)^{d-1}\\
&= \left( 2^{2rd+3d-2}d^{d-1} |\det B|^{-1} (\log 2)^{1-d} 
\sum\limits_{k=0}^{\infty} \left(2^{1-2r}\right)^k (k+1)^{d-1} \right)
\cdot a^{-2rd}\cdot (\log a)^{d-1}
.\end{split}\]
This is the desired estimate, since $2^{1-2r}<1$.
\end{proof}

Combining Theorem~\ref{keyprop} and Lemma~\ref{intlemmamix} yields:

\begin{thm}
\label{mixthm}
Let $B\in\IR^{d\times d}$ be a Frolov matrix and $r \in \IN$. 
Then there is a constant $c>0$ such that 
for every $a\geq 2^{1/d}$ and $f\in \Hmixcomp$
\[
\IE \abs{M_{a,B}(f)-I_d(f)} \leq\, c \,  a^{-rd-d/2} 
\, (\log a)^\frac{d-1}{2} \, \mixnorm{f}
.\]
\end{thm}

The worst case error of $M_{a,B}$ for functions in \edit{$\Hmixunitcomp$} is small, too.

\begin{thm}
\label{mixthmworstcase}
Let $B\in\IR^{d\times d}$ be a Frolov matrix and $r \in \IN$. 
Then there is a constant $c>0$ such that for
every $a\geq 2^{1/d}$ and $f\in \edit{\Hmixunitcomp}$
\[\begin{split}
\sup\limits_\omega \abs{M_{a,B}^\omega(f)-I_d(f)} \leq\, c \,  a^{-rd} 
\, (\log a)^\frac{d-1}{2} \, \edit{\mixnormunit{f}}
,\end{split}\]
where the supremum is taken over all realizations $M_{a,B}^\omega$ of $M_{a,B}$.
\end{thm}


\begin{proof}
The realizations $M_{a,B}^\omega$ of $M_{a,B}$ take the form $Q_{a\bar{u}B,v}$
for some $u\in [1,2^{1/d}]^d$ and $v\in[0,1]^d$. By Lemma~\ref{errorlemma} and
Hölder's inequality,
\[\begin{split}
&\abs{Q_{a\bar{u}B,v}(f)-I_d(f)}^2
\leq \left(\sum\limits_{m\in\IZ^d\setminus\set{0}} \left| 
\hat{f}(a\bar{u}Bm)\right|\right)^2\\
&\leq \left(\sum\limits_{m\in\IZ^d\setminus\set{0}} h_r(a\bar{u}Bm)^{-1}\right)
\cdot\left(\sum\limits_{m\in\IZ^d\setminus\set{0}} h_r(a\bar{u}Bm)\cdot 
\abs{\hat{f}(a\bar{u}Bm)}^2\right)
.\end{split}\]
The first factor of this product is bounded above by a constant multiple of $a^{-2rd} 
\cdot (\log a)^{d-1}$. This is proven similar to Lemma~\ref{intlemmamix}:\\
Let $N(\beta)=\{x\in\IR^d \mid \edit{\lfloor 2^{\beta_j-1}\rfloor} \leq|x_j|<2^{\beta_j},1
\leq j\leq d\}$
for $\beta\in\IN_0^d$ and $G_a^\beta=\set{m\in\IZ^d\setminus\set{0}
\mid a\bar{u}Bm\in N(\beta)}$.
Then $\IZ^d\setminus\set{0}$ is the disjoint union of all $G_a^\beta$ over $\beta\in\IN_0^d$.
Again $G_a^\beta$ is empty for $|\beta|\leq d\log_2 a$. Otherwise,
\[
h_r(a\bar{u}Bm)\geq \prod_{j=1}^d \left(1+\edit{\lfloor2^{\beta_j-1}\rfloor}^{2r}\right) 
\geq \prod_{j=1}^d 2^{2r(\beta_j-1)} = 2^{2r(|\beta|-d)}
\]
for $m\in G_a^\beta$ and hence $h_r(a\bar{u}Bm)^{-1}\leq 2^{2r(d-|\beta|)}$, and
\[\begin{split}
\abs{G_a^\beta} \leq \abs{\set{m\in\IZ^d\setminus\set{0}\mid \abs{(Bm)_j}<
\frac{2^{\beta_j}}{a}}} \leq 2^{d+\abs{\beta}} a^{-d} +1 \leq 2^{d+1+\abs{\beta}}
a^{-d},\end{split}\]
since $B$ is a Frolov matrix. This yields
\[\begin{split}
\sum\limits_{m\in\IZ^d\setminus\set{0}} h_r(a\bar{u}Bm)^{-1}
&= \sum\limits_{\beta\in\IN_0^d} \sum\limits_{m\in G_a^\beta} h_r(a\bar{u}Bm)^{-1}
\leq \sum\limits_{|\beta|>d\log_2 a} 2^{2r(d-|\beta|)}
\cdot a^{-d}\cdot 2^{d+1+|\beta|}\\
&\leq c_1\cdot a^{-2rd} \cdot (\log a)^{d-1}
,\end{split}\]
like in Lemma~\ref{intlemmamix}.

We show that the second factor in the above inequality is bounded above by a constant 
multiple of \edit{$\mixnormunit{f}^2$}. This proves the theorem. For $x\in\IR^d$ we have
\[
h_r(x)\cdot \abs{\hat{f}(x)}^{2}
=\sum\limits_{\alpha\in\set{0,\dots,r}^d}\abs{\widehat{D^\alpha f}(x)}^2
.\]
The function $g_\alpha=D^\alpha f\circ(a\bar{u}B)^{-\top}$ has compact support
in $(a\bar{u}B)^\top[0,1]^d$. Let
$M_a=\set{k\in\IZ^d \mid k+[0,1]^d\cap (a\bar{u}B)^\top[0,1]^d\neq\emptyset}$. Then
\[\begin{split}
\abs{\widehat{D^\alpha f}(a\bar{u}Bm)}^2
&= \left| \int_{\IR^d} D^\alpha f(y)\cdot 
e^{-2\pi i \langle a\bar{u}Bm,y\rangle} \d y\right|^2\\
&= \left|\frac{1}{\det(a\bar{u}B)}\int_{\IR^d} g_\alpha(x)\cdot 
e^{-2\pi i\langle m,x\rangle} \d x\right|^2\\
&= \left|\frac{1}{\det(a\bar{u}B)}\sum\limits_{k\in M_a} \langle g_\alpha(x),
e^{2\pi i\langle m,\cdot\rangle}\rangle_{L^2\left(k+[0,1]^d\right)}\right|^2\\
&\leq \frac{\abs{M_a}}{\abs{\det(a\bar{u}B)}^2} \sum\limits_{k\in M_a} 
\left|\langle g_\alpha,e^{2\pi i\langle m,\cdot 
\rangle}\rangle_{L^2\left(k+[0,1]^d\right)}\right|^2
.\end{split}\]
Thus we obtain
\[\begin{split}
&\sum\limits_{m\in\IZ^d\setminus\set{0}} h_r(a\bar{u}Bm)\cdot \abs{\hat{f}(a\bar{u}Bm)}^2
\leq \sum\limits_{m\in\IZ^d} \sum\limits_{\alpha\in\set{0,\dots,r}^d}
\abs{\widehat{D^\alpha f}(a\bar{u}Bm)}^2\\
&\leq \frac{\abs{M_a}}{\abs{\det(a\bar{u}B)}^2} \sum\limits_{m\in\IZ^d}
\sum\limits_{\alpha\in\set{0,\dots,r}^d}
\sum\limits_{k\in M_a} \left|\langle g_\alpha,e^{2\pi i\langle m,
\cdot \rangle}\rangle_{L^2\left(k+[0,1]^d\right)}\right|^2\\
&= \frac{\abs{M_a}}{\abs{\det(a\bar{u}B)}^2} \sum\limits_{\alpha\in\set{0,\dots,r}^d} 
\sum\limits_{k\in M_a} \left|\left| g_\alpha\right|\right|_{L^2\left(k+[0,1]^d\right)}^2
= \frac{\abs{M_a}}{\abs{\det(a\bar{u}B)}^2} 
\sum\limits_{\alpha\in\set{0,\dots,r}^d} \lnorm{g_\alpha}^2\\
&= \frac{\abs{M_a}}{\abs{\det(a\bar{u}B)}} \sum\limits_{\alpha\in\set{0,\dots,r}^d} 
\lnorm{D^\alpha f}^2
= \frac{\abs{M_a}}{\abs{\det(a\bar{u}B)}} \edit{\mixnormunit{f}}^2
.\end{split}\]
Since both $\abs{M_a}$ and $\abs{\det(a\bar{u}B)}$ are of order $a^d$, this yields the 
statement.
\end{proof}

\subsection*{Error Bounds for $\Hisocomp$}

If, however, $s \in \IN$ with $s > d/2$ and the integrand is from $\Hisocomp\subseteq\Cc$, 
the following lemma holds.

\begin{lemma}
\label{intlemmaiso}
For any Frolov matrix $B\in\IR^{d\times d}$ and $s\in\IN$ with $s>d/2$ there is some $c>0$
such that for each $a>0$ and $f\in\Hisocomp$
\[
\int_{D_a} \abs{\hat{f}(x)} \, \d x  \leq c \,  a^{-s+d/2} \, \isonorm{f}
.\]
\end{lemma}

\begin{proof}
Like in Lemma~\ref{intlemmamix}, we apply Hölder's inequality and get
\[\begin{split}
&\left( \int_{D_a} \abs{\hat{f}(x)} \, \d x  \right)^2
=\left( \int_{D_a} \abs{\hat{f}(x)} v_s(x)^{1/2} 
\cdot v_s(x)^{-1/2} \, \d x  \right)^2\\
&\leq \left(\int_{D_a} v_s(x)^{-1} \, \d x \right)\cdot \isonorm{f}^2
\leq \tilde{c}\cdot \left(\int_{D_a} \left(1+\Vert x\Vert_2^2
\right)^{-s} \, \d x \right)\cdot \isonorm{f}^2
,\end{split}\]
for some $\tilde{c}>0$. Since $\Vert x\Vert_2\geq a$ for $x\in D_a$, the 
latter integral is bounded by
\[\begin{split}
&\int\limits_{\set{x\in\IR^d:\,\Vert x\Vert_2\geq a}} 
\left(1+\Vert x\Vert_2^2\right)^{-s} \, \d x
= \int\limits_{a}^{\infty}\int\limits_{S_{d-1}} 
\braces{1+R^2}^{-s}\cdot R^{d-1} \, \d \sigma \, \d R \\
&= \sigma\left(S_{d-1}\right) \int\limits_{a}^{\infty} 
\braces{1+R^2}^{-s}\cdot R^{d-1} \, \d R 
\leq \sigma\left(S_{d-1}\right) \int\limits_{a}^{\infty} R^{-2s+d-1} \, \d R 
\leq \bar{c} \cdot a^{-2s+d}
,\end{split}\]
for some $\bar{c}>0$, since $-2s+d-1<-1$.
\end{proof}

In this case, combining Theorem~\ref{keyprop} and Lemma~\ref{intlemmaiso} yields:

\begin{thm}
\label{isothm}
Let $B\in\IR^{d\times d}$ be a Frolov matrix, $s \in \IN$ and $s > d/2$.
Then there is a constant $c>0$ such that for 
every $a>0$ and $f\in \Hisocomp$
\[\begin{split}
\IE \abs{M_{a,B}(f)-I_d(f)} \leq\, c \,  a^{-s-d/2} \, \isonorm{f}
.\end{split}\]
\end{thm}

\edit{The Frolov property of $B$ is important for 
Theorem~1 and the class $\Hmixunitcomp$, but} 
we remark that $B$ does not have to be a Frolov matrix to get this 
estimate on $\Hisocomp$. As seen in the proof of Lemma~\ref{intlemmaiso}, 
we do not need that the lattice points of $aB\IZ^d\setminus\set{0}$ 
lie in $D_a$ but only that they lie outside the 
ball $\set{x\in\IR^d\mid \Vert x\Vert_2\leq a}$. For example, the 
identity matrix would do. But if $B$ is a Frolov matrix, $M_{a,B}$ works 
universally for $\Hmixcomp$ and $\Hisocomp$. Furthermore, the Frolov properties
of $B$ prevent  
       extremely large jumps of the number of nodes of
$M_{a,B}=Q_{a\bar{u}B,v}$ for small changes of $a>0$ or $u\in[1,2^{1/d}]^d$.

For functions from \edit{$\Hisounitcomp$} the worst case error of $M_{a,B}$ is also small.

\begin{thm}
\label{isothmworstcase}
Let $B\in\IR^{d\times d}$ be a Frolov matrix 
and $s \in \IN$ with $s> d/2$. Then there is a constant $c>0$ such that for 
every $a>0$ and $f\in \edit{\Hisounitcomp}$
\[\begin{split}
\sup\limits_\omega \abs{M_{a,B}^\omega(f)-I_d(f)} \leq\, c \,  a^{-s}
\, \edit{\isonormunit{f}}
,\end{split}\]
where the supremum is taken over all realizations $M_{a,B}^\omega$ of $M_{a,B}$.
\end{thm}

\begin{proof}
The realizations $M_{a,B}^\omega$ of $M_{a,B}$ take the form $Q_{a\bar{u}B,v}$
for some $u\in [1,2^{1/d}]^d$ and $v\in[0,1]^d$. By Lemma~\ref{errorlemma} and
Hölder's inequality,
\[\begin{split}
&\abs{Q_{a\bar{u}B,v}(f)-I_d(f)}^2
\leq \left(\sum\limits_{m\in\IZ^d\setminus\set{0}} 
\left| \hat{f}(a\bar{u}Bm)\right|\right)^2\\
&\leq \left(\sum\limits_{m\in\IZ^d\setminus\set{0}} v_s(a\bar{u}Bm)^{-1}\right)
\cdot\left(\sum\limits_{m\in\IZ^d\setminus\set{0}} v_s(a\bar{u}Bm)
\cdot \abs{\hat{f}(a\bar{u}Bm)}^2\right)
.\end{split}\]
The first factor of this product is bounded above by a constant multiple of 
$a^{-2s}$: Since
\[
v_s(a\bar{u}Bm) \geq \left|\left|a\bar{u}Bm\right|\right|_2^{2s}
\geq a^{2s} \cdot \left|\left|Bm\right|\right|_2^{2s}
\geq a^{2s}\cdot \left|\left|B^{-1}\right|\right|_2^{-2s}\cdot
\left|\left|m\right|\right|_2^{2s}
,\]
we have
\[
\sum\limits_{m\in\IZ^d\setminus\set{0}} v_s(a\bar{u}Bm)^{-1}
\leq a^{-2s}\cdot \left|\left|B^{-1}\right|\right|_2^{2s}\cdot
\sum\limits_{m\in\IZ^d\setminus\set{0}}\left|\left|m\right|\right|_2^{-2s}
,\]
where this last series converges for $2s>d$.\\
We show that the second factor in the above inequality is bounded above 
by a constant multiple of \edit{$\isonormunit{f}^2$}. This proves the theorem.\\
For any $x\in\IR^d$ we have
\[
v_s(x)\cdot \abs{\hat{f}(x)}^{2}
=\sum\limits_{\left|\left|\alpha\right|\right|_1\leq s}\abs{\widehat{D^\alpha f}(x)}^2
.\]
The function $g_\alpha=D^\alpha f\circ(a\bar{u}B)^{-\top}$ has compact 
support in $(a\bar{u}B)^\top[0,1]^d$. Let
$M_a=\set{k\in\IZ^d \mid k+[0,1]^d\cap (a\bar{u}B)^\top[0,1]^d\neq\emptyset}$. Then
\[\begin{split}
\abs{\widehat{D^\alpha f}(a\bar{u}Bm)}^2
&= \left| \int_{\IR^d} D^\alpha f(y)\cdot
e^{-2\pi i \langle a\bar{u}Bm,y\rangle} \d y\right|^2\\
&= \left|\frac{1}{\det(a\bar{u}B)}\int_{\IR^d} g_\alpha(x)\cdot 
e^{-2\pi i\langle m,x\rangle} \d x\right|^2\\
&= \left|\frac{1}{\det(a\bar{u}B)}\sum\limits_{k\in M_a} 
\langle g_\alpha(x),e^{2\pi i\langle m,
\cdot \rangle}\rangle_{L^2\left(k+[0,1]^d\right)}\right|^2\\
&\leq \frac{\abs{M_a}}{\abs{\det(a\bar{u}B)}^2} \sum\limits_{k\in M_a} 
\left|\langle g_\alpha,e^{2\pi i\langle m,
\cdot \rangle}\rangle_{L^2\left(k+[0,1]^d\right)}\right|^2
.\end{split}\]
Thus we obtain
\[\begin{split}
&\sum\limits_{m\in\IZ^d\setminus\set{0}} v_s(a\bar{u}Bm)\cdot \abs{\hat{f}(a\bar{u}Bm)}^2
\leq \sum\limits_{m\in\IZ^d} \sum\limits_{\left|\left|\alpha\right|\right|_1\leq s}
\abs{\widehat{D^\alpha f}(a\bar{u}Bm)}^2\\
&\leq \frac{\abs{M_a}}{\abs{\det(a\bar{u}B)}^2} \sum\limits_{m\in\IZ^d}
\sum\limits_{\left|\left|\alpha\right|\right|_1\leq s}
\sum\limits_{k\in M_a} \left|\langle g_\alpha,e^{2\pi i\langle m,
\cdot \rangle}\rangle_{L^2\left(k+[0,1]^d\right)}\right|^2\\
&= \frac{\abs{M_a}}{\abs{\det(a\bar{u}B)}^2} \sum\limits_{\left|\left|\alpha\right|\right|_1
\leq s} 
\sum\limits_{k\in M_a} \left|\left| g_\alpha\right|\right|_{L^2\left(k+[0,1]^d\right)}^2
= \frac{\abs{M_a}}{\abs{\det(a\bar{u}B)}^2} \sum\limits_{\left|\left|\alpha\right|\right|_1
\leq s} 
\lnorm{g_\alpha}^2\\
&= \frac{\abs{M_a}}{\abs{\det(a\bar{u}B)}} \sum\limits_{\left|\left|\alpha\right|\right|_1
\leq s} 
\lnorm{D^\alpha f}^2
= \frac{\abs{M_a}}{\abs{\det(a\bar{u}B)}} \edit{\isonormunit{f}}^2
.\end{split}\]
Since both $\abs{M_a}$ and $\abs{\det(a\bar{u}B)}$ are of order $a^d$, this yields the 
statement.
\end{proof}

\section{The Method $\widetilde{M}_{a,B}$ for Integration on $\Hisounit$ and $\Hmixunit$}  
\label{s5}

We can transform the Monte Carlo method $M_{a,B}$ from above 
such that it is still unbiased and
\edit{its error satisfies the
same upper bounds for the full spaces $\Hmixunit$ and $\Hisounit$,
that $M_{a,B}$ satisfies for the subspaces $\Hmixunitcomp$ and $\Hisounitcomp$.}
This is done by a standard method, which is also used for deterministic 
quadrature rules for $\Hmixunitcomp$, see \cite[pp.\,359]{temlyakov}.

To that end let $\psi:\IR\to\IR$ be an infinitely differentiable function 
such that $\psi|_{(-\infty,0)}=0$, $\psi|_{(1,\infty)}=1$ 
and $\psi|_{(0,1)}:(0,1)\to(0,1)$ is a diffeomorphism. For example, we can choose
\[
h(x)=\begin{cases}
e^\frac{1}{(2x-1)^2-1} & \text{if } x\in(0,1),\\
0 & \text{else,}
\end{cases}
\quad\quad
\psi(x)=\frac{\int_{-\infty}^x h(t) \, \d t}{\int_{-\infty}^{\infty} h(t) \, \d t}
\]
for $x\in\IR$. Like $h$ also $\psi$ is infinitely differentiable and obviously 
satisfies $\psi|_{(-\infty,0)}=0$ and $\psi|_{(1,\infty)}=1$. 
Since the derivative of $\psi$ is strictly positive on $(0,1)$, 
it is strictly increasing and a bijection of $(0,1)$ with a smooth inverse function.

\begin{minipage}[h!]{.49\linewidth}
\includegraphics[width=\linewidth]{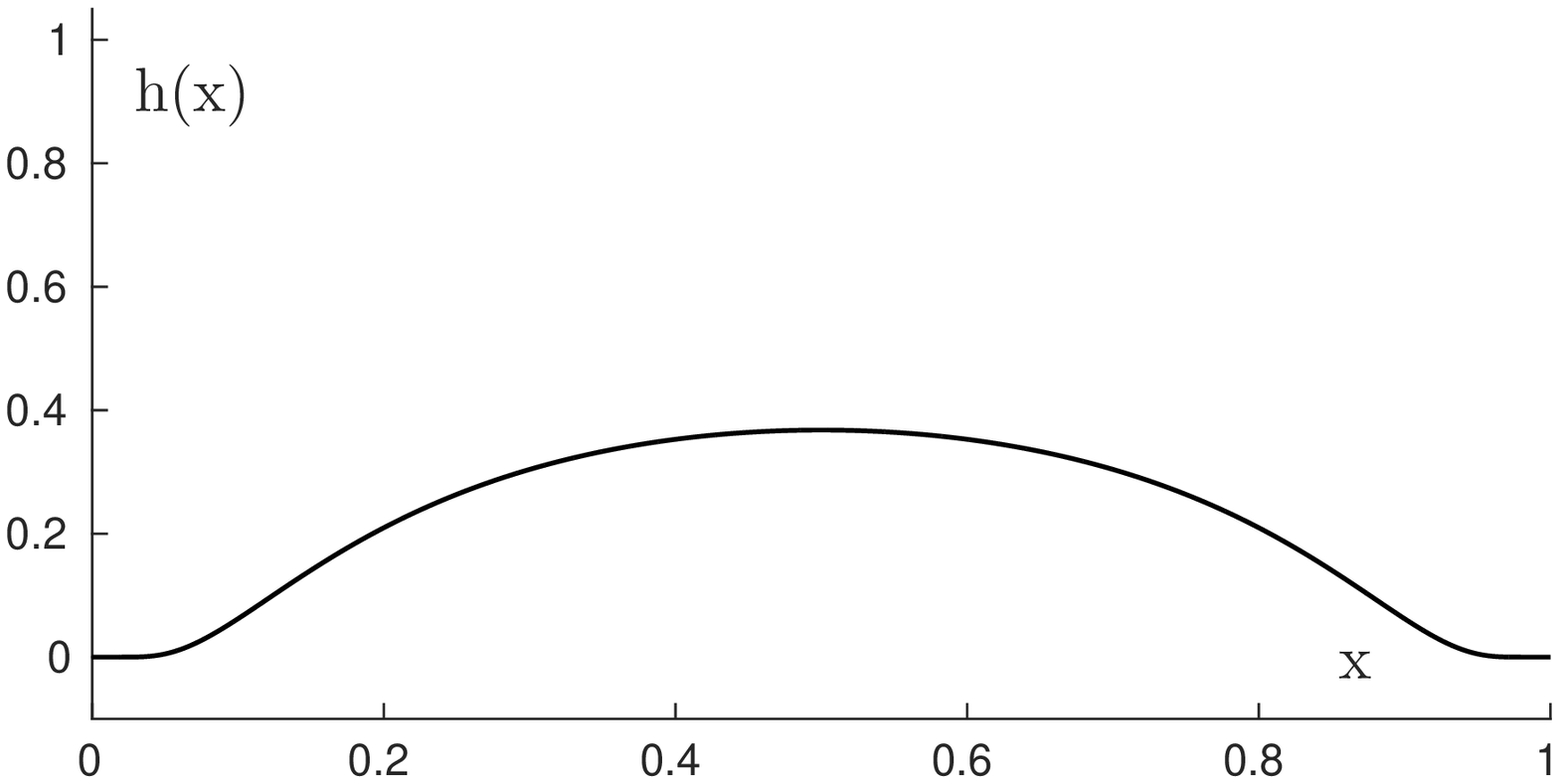}
\end{minipage}
\begin{minipage}[h!]{.46\linewidth}
\includegraphics[width=\linewidth]{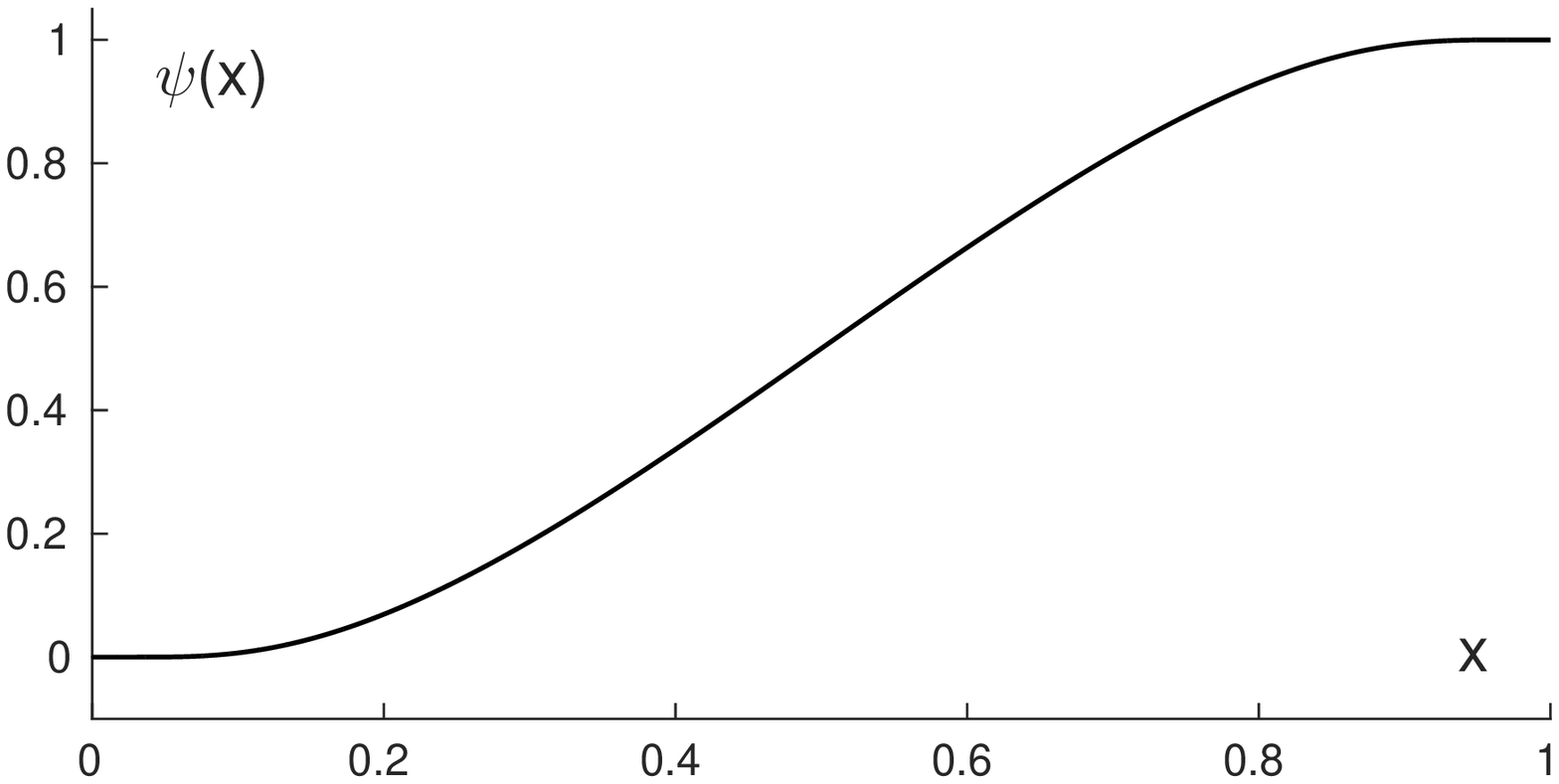}
\end{minipage}

Given such $\psi$, the map $\Psi:\IR^d\to\IR^d$ with 
$\Psi(x)=(\psi(x_1),\hdots,\psi(x_d))^{\top}$ is a diffeomorphism on $(0,1)^d$ 
with inverse $\Psi^{-1}(x)=(\psi^{-1}(x_1),\hdots,\psi^{-1}(x_d))^{\top}$ 
and $|D\Psi(x)|\overset{\psi'\geq 0}{=}\det D\Psi(x)=\prod\limits_{i=1}^{d}\psi'(x_i)$.

If $Q$ is a linear quadrature formula for integration on the unit cube
with nodes $x^{(j)}\in[0,1]^d$ and weights $a_j\in\IR$, 
where $j=1,\hdots,n$, we define the transformed quadrature formula 
$\widetilde{Q}$ by choosing the nodes and weights to be
\[
\tilde{x}^{(j)}=\Psi(x^{(j)})$ \text{\ \ \ \ and \ \ \ \ } 
$\tilde{a_j}=a_j\cdot|D\Psi(x^{(j)})|.
\]
Thus, $\widetilde{Q}_{S,v}$ for $v\in\IR^d$ and \edit{invertible} 
$S\in\IR^{d\times d}$ takes the form
\[
\widetilde{Q}_{S,v}(f)=\frac{1}{\abs{\det S}}
\sum\limits_{m\in\IZ^d} f\left(\Psi\left(S^{-\top}(m+v)\right)\right)
\cdot \left|D\Psi\left(S^{-\top}(m+v)\right)\right|
\]
for any function $f:[0,1]^d\to \IR$. Notice that $\left|D\Psi
\left(S^{-\top}(m+v)\right)\right|$ is zero, if $S^{-\top}(m+v)\not\in [0,1]^d$.

\medskip 
\noindent
{\bf Algorithm.} 
For a Frolov matrix $B\in\IR^{d\times d}$ and any $a>0$
the transformed randomized Frolov quadrature formula
$\widetilde{M}_{a,B}$
is the method $\widetilde{Q}_{a\bar{u}B,v}$ with independent $u$ and $v$, 
uniformly distributed in $[1,2^{1/d}]^d$ and $[0,1]^d$ respectively.

\medskip

\begin{lemma}
\label{MSchlangeunbiased}
The method $\widetilde{M}_{a,B}$ is well-defined and unbiased on $L^1([0,1]^d)$.
\end{lemma}

\begin{proof}
Let $f\in L^1([0,1]^d)$. \edit{By the change of variables theorem
$f_0=f\circ \Psi \cdot \left|D\Psi\right|$} is also integrable on $[0,1]^d$
and satisfies
\[
\widetilde{Q}_{a\bar{u}B,v}(f) = Q_{a\bar{u}B,v}(f_0) \quad \text{and}\quad I_d(f)=I_d(f_0)
.\]
Thus
$\widetilde{M}_{a,B}^\omega(f) = M_{a,B}^\omega(f_0)$
for any realization $\widetilde{M}_{a,B}^\omega$ of $\widetilde{M}_{a,B}$.
This yields
\[
\IE \left(\widetilde{M}_{a,B}(f)\right) = \IE \left( M_{a,B}(f_0)\right) = I_d(f_0) = I_d(f)
\]
by Lemma~\ref{Munbiased}.
\end{proof}

The following is our main result. It is important to recall that the number $n$ 
of function evaluations in $\widetilde{M}_{a,B}$ is of the order $a^d$, see Lemma~\ref{anlemma}.

\begin{thm}
\label{modthm}
Let $B\in\IR^{d\times d}$ be a Frolov matrix
and $r,s \in \IN$ with $s>d/2$. Then there is a 
constant $c>0$ such that for every $a\geq 2^{1/d}$ and $f\in \Hmixunit$
\[\begin{split}
&\IE \abs{\widetilde{M}_{a,B}(f)-I_d(f)} &&\leq\, c \,  a^{-rd-d/2} 
\, (\log a)^\frac{d-1}{2} \, \mixnormunit{f}\quad \text{and}\\
&\sup\limits_\omega \abs{\widetilde{M}_{a,B}^\omega(f)-I_d(f)} &&\leq\, c \,  a^{-rd} 
\, (\log a)^\frac{d-1}{2} \, \mixnormunit{f}
\end{split}\]
and for every $a>0$ and $f\in \Hisounit$
\[\begin{split}
&\IE \abs{\widetilde{M}_{a,B}(f)-I_d(f)} &&\leq\, c \, 
a^{-s-d/2} \, \isonormunit{f}\quad \text{and}\\
&\sup\limits_\omega \abs{\widetilde{M}_{a,B}^\omega(f)-I_d(f)} &&\leq\, c \,  a^{-s}
\, \isonormunit{f}
,\end{split}\]
where the suprema are taken over all 
realizations $\widetilde{M}_{a,B}^\omega$ of $\widetilde{M}_{a,B}$.
\end{thm}

\begin{proof}
Remember that
$\widetilde{M}_{a,B}^\omega(f) = M_{a,B}^\omega(f_0)$
for $f\in L^1(\IR^d)$, $f_0=f\circ \Psi \cdot \left|D\Psi\right|$  
and any realization $\widetilde{M}_{a,B}^\omega$ of $\widetilde{M}_{a,B}$. 
Since $\psi'(x)=0$ for $x\not\in(0,1)$, we know that 
$\diff^\alpha f_0|_{\partial[0,1]^d}=0$ for each $\alpha\in\{0,\dots,r\}^d$ 
and hence $f_0\in \Hmixunitcomp \subseteq \Hmixcomp$ for each $f\in\Hmixunit$.

This yields
\[\begin{split}
\IE \abs{\widetilde{M}_{a,B}(f)-I_d(f)} = \IE \abs{ M_{a,B}(f_0)-I_d(f_0)}
\leq c \cdot a^{-rd-d/2} (\log a)^\frac{d-1}{2} \cdot \mixnorm{f_0}
\end{split}\]
as well as
\[\begin{split}
\sup\limits_\omega \abs{\widetilde{M}_{a,B}^\omega(f)-I_d(f)} 
= \sup\limits_\omega \abs{ M_{a,B}^\omega(f_0)-I_d(f_0)}
\leq c \cdot a^{-rd} (\log a)^\frac{d-1}{2} \cdot \mixnorm{f_0}
,\end{split}\]
if $c>0$ is the maximum of the constants of 
Theorem~\ref{mixthm} and Theorem~\ref{mixthmworstcase}. That proves the first statement, 
since there is a constant $c_0>0$ such that every \edit{function} $f\in\Hmixunit$ 
satisfies $\mixnorm{f_0}\leq c_0\, \mixnormunit{f}$. 

This can be proven as follows.
The partial derivatives of $f_0$ take the form
\[
D^\alpha f_0(x)=\frac{\partial^{\left\Vert \alpha\right\Vert_1}}{\partial 
x_1^{\alpha_1}\cdots\partial x_d^{\alpha_d}}\ f(\Psi(x))
\cdot\prod\limits_{i=1}^{d}\psi'(x_i) 
= \sum\limits_{\beta_1,\hdots,\beta_d=0}^{\alpha_1,\hdots,\alpha_d} 
D^\beta f(\Psi(x))\cdot S_{\alpha,\beta}(x)
\]
for $\alpha\in\{0,1,\hdots,r\}^d$, where $S_{\alpha,\beta}(x)$ is a finite 
sum of finite products of terms $\psi^{(j)}(x_i)$ with $i\in\{1,\hdots,d\}, 
j\in\{1,\hdots,rd+1\}$ and does not depend on $f$. It is therefore continuous 
and bounded by some $c_{\alpha,\beta}>0$. Using the Cauchy inequality
\[
\left|\sum\limits_{i=1}^{\dim v}v_i\right|^2=\left|\langle v,(1,\hdots,1)^{\top}\rangle 
\right|^2
\leq \left\Vert v\right\Vert_2^2\cdot\left\Vert(1,\hdots,1)^{\top}\right\Vert_2^2
=\dim v \cdot \sum\limits_{i=1}^{\dim v}|v_i|^2
\]
for real vectors $v$, we get
\[\begin{split}
\lnorm{D^\alpha f_0}^2
&\leq \left(\sum\limits_{\beta_1,\hdots,\beta_d=0}^{\alpha_1,\hdots,\alpha_d} 
\llnorm{(D^\beta f \circ \Psi) \cdot S_{\alpha,\beta}}\right)^2\\
&\leq \left(\sum\limits_{\beta_1,\hdots,\beta_d=0}^{\alpha_1,\hdots,\alpha_d} 
c_{\alpha,\beta} \cdot\llnorm{D^\beta f \circ \Psi}\right)^2\\
&\overset{\text{Cauchy}}{\leq} (r+1)^d\sum\limits_{\beta_1,\hdots,\beta_d=0}^{\alpha_1,
\hdots,\alpha_d} c_{\alpha,\beta}^2 \cdot\llnorm{D^\beta f \circ \Psi}^2\\
&= (r+1)^d\sum\limits_{\beta_1,\hdots,\beta_d=0}^{\alpha_1,\hdots,\alpha_d} 
c_{\alpha,\beta}^2 \int_{(0,1)^d} |D^\beta f (\Psi (x))|^2 \,\mathrm{d}x\\
&= (r+1)^d\sum\limits_{\beta_1,\hdots,\beta_d=0}^{\alpha_1,\hdots,\alpha_d} 
c_{\alpha,\beta}^2 \int_{\Psi\left((0,1)^d\right)} |D^\beta f (\Psi (\Psi^{-1}(x))|^2 
\cdot |D\Psi^{-1}(x)| \, \mathrm{d}x\\
&\leq (r+1)^d\sup\limits_{x\in(0,1)^d} |D\Psi^{-1}(x)| 
\sum\limits_{\beta_1,\hdots,\beta_d=0}^{\alpha_1,\hdots,\alpha_d} c_{\alpha,\beta}^2 
\cdot \llnorm{D^\beta f}^2\\
&\leq c_\alpha \cdot \mixnormunit{f}^2\
,\end{split}\]
for some $c_\alpha>0$ and
\[
\mixnorm{f_0}^2
=\sum\limits_{\alpha\in\{0,1,\hdots,r\}^d}\lnorm{D^\alpha f_0}^2
\leq \left(\sum\limits_{\alpha\in\{0,1,\hdots,r\}^d} c_\alpha\right)\cdot \mixnormunit{f}^2
.\]

The second statement is proven in the exact same manner.
\end{proof}

A translation of Theorem~\ref{modthm} by means of Lemma~\ref{anlemma}
shows that the algorithm
$\widetilde{M}_{a,B}$ indeed satisfies 
all the properties 
{\bf (P1)}  and  {\bf (P2)}  and {\bf (P3)}.

\begin{summary}
Let $d,r,s\in \IN$ with $s>d/2$ and $B\in\IR^{d\times d}$ be a Frolov matrix.
Then there is a constant $c>0$ 
\edit{(that may depend on $B$ and $r$ or $s$)} 
such that for every $n\in\IN$ there is some
$a_n>0$ so that $\widetilde{M}_{a_n,B}$ uses at most $n$ function values
of any $f\in L^1([0,1]^d)$ and satisfies
\[\begin{split}
&\IE \left(\widetilde{M}_{a_n,B}(f)\right) &&= I_d(f),\\
&\IE \abs{\widetilde{M}_{a_n,B}(f)-I_d(f)} &&\leq\, c \,  n^{-r-1/2} 
\, (\log n)^\frac{d-1}{2} \, \mixnormunit{f},\quad \text{if n>1,}\\
&\sup\limits_\omega \abs{\widetilde{M}_{a_n,B}^\omega(f)-I_d(f)} &&\leq\, c \,  n^{-r} 
\, (\log n)^\frac{d-1}{2} \, \mixnormunit{f},\quad \text{if n>1,}\\
&\IE \abs{\widetilde{M}_{a_n,B}(f)-I_d(f)} &&\leq\, c \, 
n^{-s/d-1/2} \, \isonormunit{f},\quad \text{and}\\
&\sup\limits_\omega \abs{\widetilde{M}_{a_n,B}^\omega(f)-I_d(f)} &&\leq\, c \,  n^{-s/d}
\, \isonormunit{f}
.\end{split}\]
\end{summary}

\begin{proof}
Let $c_1=(\Vert B\Vert_1+1)^d\geq 1$ 
be the constant of Lemma~\ref{anlemma} and $c_2>0$ be the constant of 
Theorem~\ref{modthm}. For $n\geq 4c_1$, we set $a_n=\left(n/(2c_1)\right)^{1/d}\geq 2^{1/d}$.
By Lemma~\ref{anlemma}, the Monte Carlo method $\widetilde{M}_{a_n,B}$ uses no more than
$2c_1\cdot a_n^d=n$ function values of $f$.
For $n<4c_1$ we choose $a_n>0$ small enough such that the only node of
$\widetilde{M}_{a_n,B}$ is zero. The method $\widetilde{M}_{a_n,B}$ is thus unbiased
for all $n\in\IN$ and uses at most $n$ function values.\\
Theorem~\ref{modthm} yields for $n\geq 4c_1$ and thus $a_n\geq 2^{1/d}$ that
\[\begin{split}
\IE \abs{\widetilde{M}_{a_n,B}(f)-I_d(f)}
&\leq c_2 \, \left(\left(\sfrac{n}{2c_1}\right)^{1/d}
\right)^{-rd-d/2}\left(\log \left(\sfrac{n}{2c_1}
\right)^{1/d}\right)^{\frac{d-1}{2}} \, \mixnormunit{f}\\
&\overset{c_1\geq 1}{\leq}  c_2 \, d^{-\frac{d-1}{2}}\, (2c_1)^{r+1/2} 
\, n^{-r-1/2}\left(\log n\right)^{\frac{d-1}{2}}\,\mixnormunit{f},\\
\IE \abs{\widetilde{M}_{a_n,B}(f)-I_d(f)}
&\leq c_2 \, \left(\left(\sfrac{n}{2c_1}\right)^{1/d}
\right)^{-s-d/2}\,\isonormunit{f}\\
&=  c_2 \,  (2c_1)^{s/d+1/2} \, n^{-s/d-1/2}\,\isonormunit{f},\\
\sup\limits_\omega \abs{\widetilde{M}_{a_n,B}^\omega(f)-I_d(f)}
&\leq c_2 \, \left(\left(\sfrac{n}{2c_1}\right)^{1/d}
\right)^{-rd}\left(\log \left(\sfrac{n}{2c_1})
\right)^{1/d}\right)^{\frac{d-1}{2}} \, \mixnormunit{f}\\
&\overset{c_1\geq 1}{\leq}  c_2 \, d^{-\frac{d-1}{2}}\, (2c_1)^{r} \,  
n^{-r} \, (\log n)^\frac{d-1}{2} \, \mixnormunit{f},\\
\sup\limits_\omega \abs{\widetilde{M}_{a_n,B}^\omega(f)-I_d(f)}
&\leq c_2 \, \left(\left(\sfrac{n}{2c_1}\right)^{1/d}
\right)^{-s}\, \isonormunit{f}\\
&=  c_2 \, (2c_1)^{s/d} \,  n^{-s/d} \, \isonormunit{f}
.\end{split}\]
This shows that the stated bounds hold for the maximum $c$ of the constants
$c_2 \, d^{-\frac{d-1}{2}}\, (2c_1)^{r+1/2}$, $c_2 \,  (2c_1)^{s/d+1/2}$ and possibly
larger constants that result from the cases $n=1,\dots,\lfloor 4c_1\rfloor$.
\end{proof}


\begin{thebibliography}{XX}
\addcontentsline{toc}{section}{References}



\bibitem[1]{frolov} K.\,K.\,Frolov: \emph{Upper Error Bounds for 
Quadrature Formulas on Function Classes}. 
Soviet Mathematics Doklady \textbf{17/6}, pp.\,1665--1669, 1976.

\bibitem[2]{hickernell} F.\,J.\,Hickernell:
\emph{My dream quadrature rule}.
J. Complexity {\bf 19}, pp.\,420--427, 2003.

\bibitem[3]{koch} H.\,Koch: \emph{Number Theory: 
Algebraic Numbers and Functions}. Graduate Studies in Mathematics. 
American Mathematical Society, Providence, 2000.

\bibitem[4]{ln} E.\,Novak: 
\emph{Deterministic and Stochastic Error Bounds in 
Numerical Analysis}. 
Lecture Notes in Mathematics 1349, Springer, 1988. 

\bibitem[5]{novwoz} E.\,Novak and H.\,Wo\'zniakowski:
\emph{Tractability of Multivariate Problems II: 
Standard Information for Functionals}. 
European Mathematical Society, 2010. 


\bibitem[6]{skriganov} M.\,M.\,Skriganov:
\emph{Constructions of uniform distributions in terms of geometry 
of numbers}.  
Algebra i Analiz 6, 200--230, 1994. 

\bibitem[7]{temlyakovbuch} V.\,N.\,Temlyakov: 
\emph{Approximation of Periodic Functions}. Computational 
Mathematics and Analysis Series. Nova Science Publishers, New York, 1993.

\bibitem[8]{temlyakov} V.\,N.\,Temlyakov: \emph{Cubature formulas, discrepancy, 
and nonlinear approximation}. Journal of Complexity \textbf{19}, pp.\,352--391, 2003.

\bibitem[9]{ullrich} M.\,Ullrich: \emph{On "Upper error bounds for quadrature 
formulas on function classes" by K.\,K.\,Frolov}. 
Manuscript, http://arxiv.org/abs/1404.5457, 2014.

\bibitem[10]{marioandtino} M.\,Ullrich and T.\,Ullrich:
\emph{The role of Frolov's cubature formula for 
functions with bounded mixed derivative}. 
Manuscript, http://arxiv.org/abs/1503.08846, 2015.


\end{thebibliography}
\end{document}